\documentclass[12pt]{amsart}

\usepackage{amstext}
\usepackage{amsthm}
\usepackage{amsmath}
\usepackage{amssymb}
\usepackage{latexsym}
\usepackage{amsfonts}
\usepackage{color}
\usepackage{graphicx}
\usepackage{yfonts}
\usepackage{verbatim}
\usepackage{enumitem}
\usepackage{tikz-cd}
\numberwithin{equation}{section}

\AtBeginDocument{%
 \def\MR#1{} %Get rid of MR numbers in references
}

\usepackage[pagebackref,hypertexnames=false, colorlinks, citecolor=black, linkcolor=black, urlcolor=red]{hyperref}
\usepackage[backrefs]{amsrefs}

\usepackage{latexsym}
\usepackage{amssymb}
\usepackage{euscript}
\let\cal=\mathcal      

\def\mcc{M\raise.5ex\hbox{c}C}
\def\mccarthy{M\raise.5ex\hbox{c}Carthy}
\def\Hu{\H}
\def\sz{Szeg\Hu{o} }

\def\M{{\cal M}}

\let\i=\infty

\def\la{\langle}
\def\ra{\rangle}
\def\={\ = \ }

\def\C{\mathbb C}

\def\T{\mathbb T}

\def\B{\mathbb B}

\def\be{\setcounter{equation}{\value{theorem}} \begin{equation}}
\def\ee{\end{equation} \addtocounter{theorem}{1}}
\def\beq{\begin{eqnarray*}}
\def\eeq{\end{eqnarray*}}

\def\bp{{\sc Proof: }}
\def\ep{{}{\hfill $\Box$} \vskip 5pt \par}

\def\bl{\begin{lemma}}
\def\el{\end{lemma}}
\def\bt{\begin{theorem}}
\def\et{\end{theorem}}
\def\bprop{\begin{prop}}
\def\eprop{\end{prop}}
\def\bd{\begin{definition}}
\def\ed{\end{definition}}
\def\br{\begin{remark}}
\def\er{\end{remark}}
\def\bexer{\begin{exercise}}
\def\eexer{\end{exercise}}

\newtheorem{theorem}{Theorem}[section]
\newtheorem{proposition}[theorem]{Proposition}
\newtheorem{lemma}[theorem]{Lemma}
\newtheorem{corollary}[theorem]{Corollary}

\newtheorem{definition}[theorem]{Definition}

\newtheorem{question}[theorem]{Question}

\newtheorem{defin}[theorem]{Definition}
\newtheorem{example}[theorem]{Example}

\renewcommand\Re{\mathrm{Re\, }}

\newcommand{\cH}{\mathcal H}

\newcommand{\bB}{\mathbb B}
\newcommand{\bD}{\mathbb D}
\newcommand{\ol}[1]{\overline{#1}}

\DeclareMathOperator{\Mult}{Mult}

\newcommand\mm{{\rm Mult}(\M)}
\newcommand\mom{\M \odot \M}
\newcommand\chp{{\mathcal H}^p}
\newcommand\cho{{\mathcal H}^1}
\newcommand\cht{{\mathcal H}^2}

\newcommand\chq{{\mathcal H}^q}
\newcommand\Han{{\rm Han}}

\newcommand\hha{{\mathbb H}_{1/2}}
\newcommand\dd{d}

\newcommand{\rankone}[2]{\langle \cdot , #1 \rangle #2}

\title{An $H^p$ scale for complete Pick spaces}
\date{\today}
\author[A. Aleman]{Alexandru Aleman}
\address{Lund University, Mathematics, Faculty of Science, P.O. Box 118, S-221 00 Lund, Sweden}
\email{alexandru.aleman@math.lu.se}

\author[M. Hartz]{Michael Hartz}
\address{Fachrichtung Mathematik, Universit\"at des Saarlandes, 66123 Saarbr\"ucken, Germany}
\email{hartz@math.uni-sb.de}
\thanks{M.H. was partially supported by a GIF grant.}

\author[J. M\raise.5ex\hbox{c}Carthy]{John E. M\raise.5ex\hbox{c}Carthy}
\address{Department of Mathematics, Washington University in St. Louis, One Brookings Drive,
	St. Louis, MO 63130, USA}
\email{mccarthy@wustl.edu}
\thanks{J.M. was partially supported by National Science Foundation Grant DMS 1565243.}

\author[S. Richter]{Stefan Richter}
\address{Department of Mathematics, University of Tennessee, 1403 Circle Drive, Knoxville, TN 37996-1320, USA}
\email{srichter@utk.edu}

\begin{document}

\begin{abstract}
  We define by interpolation a scale analogous to the Hardy $H^p$ scale for complete Pick spaces,
   and establish some of the basic properties
  of the resulting spaces, which we call  $\mathcal{H}^p$.
  In particular, we obtain an $\mathcal{H}^p-\mathcal{H}^q$ duality and establish
  sharp pointwise estimates for functions in $\mathcal{H}^p$.
\end{abstract}

\bibliographystyle{plain}
\maketitle

\section{Introduction}
\label{seca}

Let $\M$ be a reproducing kernel Hilbert space on a set $X$, with kernel function $k$.
Let $\mm$ denote the multiplier algebra of $\M$. We shall make the following assumption
throughout our paper:
\[
(A) \qquad \mm \ {\rm is\ densely\ contained\ in\ } \M .
\]
We shall let $\mom$ denote the weak product of $\M$ with itself, which is
\be
\label{eqa1}
\mom  \ := \ \Big\{ \sum_{n=1}^\infty f_n g_n \ : \ \sum_n \| f_n \|_\M \| g_n\|_\M < \i \Big\}.
\ee
This is a Banach space, where the norm of a function $h$ is the infimum
of $\sum_n \| f_n \|_\M \| g_n\|_\M$ over all representations of $h$ as $\sum_n f_n g_n$.

If we use the complex method of interpolation to interpolate between $\mom$ and its anti-dual
(the space of bounded conjugate linear functionals) we get a scale
of Banach function spaces, whose mid-point is the Hilbert space $\M$. By analogy with the case where $\M$ is the
Hardy space $H^2$ on the unit disk,  where the end-points become $H^1$ and BMOA and the
intermediate spaces are $H^p$ for $1 < p < \infty$, we shall define
\be
\label{eqa2}
\chp \ := \ [ \mom, (\mom)^\dagger ]_{[\theta]}
\ee
where $0 \leq \theta \leq 1$, we set
$p = \frac{1}{1-\theta}$, and $A^\dagger$ denotes the anti-dual of $A$. See \cite{bl76} for background on interpolation.

We consider $\chp$ to play the r\^ole of the $H^p$ scale for the space $\M$.
We should note that even when $k$ is the \sz kernel, the spaces $\chp$ are isomorphic but not isometric
to $H^p$ \cite{fs72}, so $\chp$ can at best be considered a renormed version of $H^p$.
In Section~\ref{secb} we collect properties of the $\chp$ spaces for general $\M$.
In Section~\ref{secc} we specialize to the case that $\M$ is a complete Pick space,
and prove that several inequalities that hold in general become equivalences in complete Pick spaces.

Our main result is the following. We let
$\delta_x$ denote the functional of evaluation at $x \in X$ and $k_x(y) = k(y,x)$ be the reproducing kernel of $\mathcal{M}$.
We will explain what a normalized complete Pick kernel is in Section \ref{secc}, and
we will explain ${\rm Han}$ and  ${\rm Han}_0$,
 the dual and  predual of $\cho$, in Section \ref{secb}.

{\bf Theorem \ref{thmc1}.}
{\em
Let $k$ be a normalized complete Pick kernel on $X$. Then for all $1 < p < \infty$ with $\frac{1}{p} + \frac{1}{q} = 1$,
  \begin{enumerate}[label=\normalfont{(\alph*)},resume]
    \item $\| \delta_x\|_{(\chp)^*} = \|k_x\|_{\chq} \approx k(x,x)^{1/p}$,
   \item $\| \delta_x\|_{(\cH^1)^*} = \| k_x\|_{{\rm Han}} = k(x,x)$,
   \item $\| \delta_x \|_{{\rm Han}_0^*} = \| \delta_x\|_{{\rm Han}^*} = \|k_x \|_{\cho} \lesssim 1 + \log(k(x,x))$,
\item for the Drury-Arveson kernel $S$,
 $\| \delta_x \|_{{\rm Han}_0^*}  \approx 1 + \log(S(x,x))$,
  \end{enumerate}
  where the implied constants
do not  depend
 on $k$ or $x$.
}

 Here and in the sequel, if $f,g:X \to [0,\infty)$ are functions, we write $f(x) \lesssim g(x)$ to mean that there exists a constant $C$ so that $f(x) \le C g(x)$
 for all $x \in X$, and $f \approx g$ if $f \lesssim g$ and $g \lesssim f$.

 We  show in Examples \ref{exam316} and \ref{exam317} that the estimate in part (c) of the theorem may not be an equivalence. In Theorem \ref{thmIS} we show that when $k$ is a normalized complete Pick kernel,
 the interpolating sequences for $\cho$ and $\cht$ coincide.

In Section \ref{sece} we close with some questions about $\chp$ scales.

\section{General Spaces}
\label{secb}

Let $\mathcal{M}$ be a reproducing kernel Hilbert space satisfying assumption (A).
The space $(\mom)^\dagger$ was described in \cite{ahmrWP}; let us recall that description.
We let $\M \otimes_\pi \M$ denote the projective tensor product of $\M$ with itself.
Its dual is ${\mathcal B}(\M, \overline{\M})$, where $\overline{\M}$ is the complex conjugate of $\M$.
Let $\rho: \M \otimes_\pi \M \to \mom$ be defined by
\[
\rho: \sum f_n \otimes g_n \mapsto \sum f_n(x) g_n(x) .
\]
Then $( \mom)^*$ can be identified with $({\rm ker} \rho)^\perp$.
We can identify $(\mom)^\dagger$ with
\[
{\rm Han} \ := \ \{ \overline{T 1} : T \in ({\rm ker} \rho)^\perp \} .
\]
If $b \in {\rm Han}$, which is a subset of $\M$, the corresponding conjugate linear functional on $\mom$ is given by
\[
\Lambda_b : f \mapsto \la b , f \ra \quad \forall \ f \in \M .
\]
We write $H_b$ for the unique operator $H \in {\mathcal B}(\M, \overline{\M}) \cap ({\rm ker} \rho)^\perp$
that satisfies $H_b 1 = \overline{b}$.
This operator is characterized by the identity
\begin{equation*}
  \langle H_b f, \overline{\phi} \rangle_{\overline{\mathcal{M}}}
  = \langle \phi f, b \rangle_{\mathcal{M}} \quad (\phi \in \Mult(\mathcal{M}), f \in \mathcal{M}).
\end{equation*}

We put a norm on ${\rm Han}$ by declaring $\| b \|$ equal to
the operator norm of $H_b$.
Let
\begin{eqnarray}
\nonumber
{\mathcal X}(\M)  \ := \ \{ b \in \M : \ \exists
\ C \geq 0 \ &{\rm s.t.\ } |\langle b , \phi f \ra | \leq
C \| \phi \|_\M \| f \|_\M \\
& \ \forall \ \phi \in \mm, f \in \M \} .
\label{eqxm}
\end{eqnarray}
Then under assumption (A) it is proved in \cite[Thm 2.5]{ahmrWP} that
\[
 {\rm Han} \subseteq {\mathcal X}(\M) .
\]
  In particular, $\Han \subseteq \mathcal{M} \odot \mathcal{M}$ contractively, so $(\mathcal{M} \odot \mathcal{M}, \Han)$
is a compatible couple of Banach spaces.
For $1 \le p < \infty$, we shall let $\chp$ be defined by
\begin{equation*}
  \mathcal{H}^p = [ \mathcal{M} \odot \mathcal{M}, \Han]_{[\theta]}
\end{equation*}
with $\theta = \frac{p-1}{p}$.
Since $\Han$ is dense in $\mathcal{M} \odot \mathcal{M}$, we have $\mathcal{H}^1 = \mathcal{M} \odot \mathcal{M}$
and $[\mathcal{H}^1,\Han]_{[1]} = \Han$ with equality of norms; see \cite[Thm 4.2.2]{bl76}.
Since we shall use it several times, let us state Calder\'on's reiteration theorem  \cite[Thm. 4.6.1]{bl76}.
\begin{theorem}
\label{reit}
Let $X_0, X_1$ be a compatible couple of complex Banach spaces with $X_1 \subseteq X_0$.
For every $0 \leq \theta \leq 1$, and $0 \leq \theta_0 \leq \theta_1 \leq 1$, let $\eta = (1-\theta) \theta_0 + \theta \theta_1$. Then
 we have
\[
[ \ [ X_0, X_1]_{[\theta_0]}, [ X_0, X_1]_{[\theta_1]} \ ]_{[\theta]}
\ = \ [ X_0, X_1]_{[\eta]}.
\]
\end{theorem}
First, we remark that the spaces $\chp$ are indeed function spaces.

\begin{proposition}
  \label{prop:basic_inclusion}
  The space $\mathcal{H}^p$ is a Banach function space on $X$ for $1 \le p < \infty$.
  Moreover, if $1 \le p \le q < \infty$, then
\[
\mathcal{M} \odot \mathcal{M} \supseteq \chp \supseteq \chq \supseteq {\rm Han}
\]
with contractive inclusions.
\end{proposition}

\begin{proof}
  Since $\Han$ is contractively contained in $\mathcal{M} \odot \mathcal{M}$,
  complex interpolation shows that
  \begin{equation*}
    \mathcal{M} \odot \mathcal{M} \supseteq \mathcal{H}^p \supseteq \Han
  \end{equation*}
  with contractive inclusions for all $1 \le p < \infty$.
  In particular, $\mathcal{H}^p$ consists of functions on $X$.
  Since point evaluations are continuous on $\mathcal{M} \odot \mathcal{M}$ and on $\Han$, they are continuous on $\mathcal{H}^p$
  for all $p$. Thus, $\mathcal{H}^p$ is a Banach function space on $X$.
  Finally, the reiteration theorem \ref{reit} shows that interpolating between $\mathcal{H}^p$ and $\Han$,
  we obtain $\mathcal{H}^q$ for $p \le q < \infty$, hence $\mathcal{H}^p \supseteq \mathcal{H}^q \supseteq \Han$
  with contractive inclusions.
\end{proof}

As one would expect, we recover the original Hilbert function space for $p=2$.

\bt
\label{thmb1}
We have $\mathcal{H}^2 = \mathcal{M}$ with equality of norms.
\et
\bp
Assumption (A) implies that $\M$ is dense in $\mom$.
G. Pisier proved in \cite{pis96} that if a Hilbert space $\M$  is densely and continuously contained in
a Banach space $A$, and so $A^\dagger$ embeds in $\mathcal{M}$, then $[ A, A^\dagger]_{[\frac{1}{2}]} = \M$,
with equality of norms.
His proof is in the context of operator spaces; a direct proof of the fact
is given in \cite{wat00}. See also \cite{cs98} for another proof.
\ep

Next, we establish the expected duality between $\mathcal{H}^p$ spaces.

\bt
\label{thm:dual}
For $1 < p < \infty $, we have $(\chp)^\dagger$ is isometrically isomorphic to $\mathcal{H}^q$, where $q$ is the conjugate index to $p$.
The action of $\mathcal{H}^q$ on $\mathcal{H}^p$ is given by the inner product of $\mathcal{H}$ on the common subspace $\Han$.
\et
\bp
By the reiteration theorem \ref{reit}, if we interpolate between $\cho$
and $\cht$ we get $\chp$ for $1 < p < 2$, and if we interpolate between $\cht$ and ${\rm Han}$
we get $\chp$ for $2 < p < \infty$.
Since $\cht$ is reflexive, we have by the duality theorem \cite[Cor. 4.5.2]{bl76} and Theorem \ref{thmb1} that
\[
[ \cho, \cht]_{[\theta]}^\dagger \= [ {\rm Han}, \cht]_{[\theta]}.
\]
(The duality theorem also applies to anti-duals because $[\overline{A_0},\overline{A_1}]_{[\theta]} =
  \overline{[A_0,A_1]_{[\theta]}}$ isometrically).
It is part of the duality theorem that the action of an element of $[\Han, \mathcal{H}^2]_{[\theta]} \subset \mathcal{H}^2$
on an element of the subspace $\mathcal{H}^2 \subset [ \mathcal{H}^1,\mathcal{H}^2]_{[\theta]}$ is given by the inner product
of $\mathcal{H}^2$; see the discussion preceding \cite[Theorem 2.7.4]{pi03}.
This proves the theorem for $1 < p \leq 2$.

In \cite[12.2]{cal64}, Calder\/on proved that if one end point space is reflexive, all the intermediate ones are too.
So this proves the theorem for $2 < p < \infty$.
\ep

 We define ${\rm Han}_0$ by
\[
{\rm Han}_0 \ := \ \{ b \in {\rm Han} : H_b \ {\rm is\ compact} \} .
\]
By \cite[Thm. 2.5]{ahmrWP}, the dual space of ${\rm Han}_0$ is $\mom$.
We think of  ${\rm Han}_0$ as the analogue of VMOA.

By \cite[Thm. 2.1]{ahmrWP}, point evaluations are in ${\rm Han}_0$.
Moreover, since $\M$ is dense in $\mom$, the point evaluations come
from pairing with the kernel functions, so each kernel function is in ${\rm Han}_0$, and
by
 the Hahn-Banach theorem,
 the set of finite linear combinations of kernel functions
is dense in $ {\rm Han}_0$.
\begin{proposition}
\label{proph0}
For $0 < \theta < 1$, we have the isometric equality
\[
[\mom,  {\rm Han}_0]_{[\theta]} \= [\mom,  {\rm Han}]_{[\theta]} .
\]
\end{proposition}
\bp
Since $\Han_0$ and $\Han$ are contained in $\mathcal{M} \odot \mathcal{M}$, so are the interpolation spaces
in the statement.
By \cite[Cor. 4.5]{cs98}, we have isometrically
\[
[\mom,  {\rm Han}_0]_{[\frac{1}{2}]} \= \M .
\]
Therefore the reiteration theorem and Theorem \ref{thmb1} prove the result for $0 < \theta \leq \frac{1}{2}$.
It remains to prove that
\[
[\M,  {\rm Han}_0]_{[s]} \=
[\M,  {\rm Han}]_{[s]}
\]
for $0 < s < 1$. But applying the duality theorem twice we get
an isometric isomorphism
\[
[\M,  {\rm Han}_0]_{[s]}^{**} \cong
[\M,  {\rm Han}]_{[s]} ,
\]
and by Calder\'on's reflexivity theorem again, we
have $[\M,  {\rm Han}_0]_{[s]}$ is reflexive for $0 \leq s < 1$, so we are done.
Using the fact that the inclusion $\Han_0 \subseteq \Han \cong (\Han_0)^{**}$ agrees
with the canonical embedding into the bidual as well as the particular form of the duality
in the duality theorem, one checks that the resulting isometric isomorphism
$[\mathcal{M},\Han]_{[s]} \cong [\mathcal{M}, \Han_0]_{[s]}$ is in fact the identity.
\ep

%{\blue Is there an easy way to see that these identifications are compatible?}

\begin{corollary}
  \label{cor:dense}
  The linear span of kernel functions is dense in $\mathcal{H}^p$ for $1 \le p < \infty$.
\end{corollary}

\begin{proof}
  It is a general result about complex interpolation that for a compatible couple of Banach spaces $(A_0,A_1)$,
    the intersetion $A_0 \cap A_1$ is dense in the intermediate interpolation spaces; see \cite[Theorem 4.2.2]{bl76}.
  Thus, Proposition \ref{proph0}  implies that $\Han_0$ is dense in $\mathcal{H}^p$ for $1 \le p < \infty$,
  In turn, finite linear combinations of kernels are dense in $\Han_0$ and the inclusion $\Han_0 \subset \mathcal{H}^p$
  is continuous.
\end{proof}

We can now show pointwise estimates that are valid in all reproducing kernel Hilbert spaces satisfying assumption (A).

\begin{proposition}
  \label{propb2}
  Let $1 < p \le 2$ and let $\frac{1}{p} + \frac{1}{q} = 1$. Then for all $x \in X$,
  \begin{enumerate}[label=\normalfont{(\alph*)}]
    \item $\| \delta_x\|_{(\cH^1)^*} = \| k_x\|_{{\rm Han}} = k(x,x)$.
    \item $\| \delta_x\|_{(\cH^p)^*} = \|k_x\|_{\cH^q} \le k(x,x)^{1/p}$,
    \item $ \| \delta_x \|_{(\cH^q)^*} = \| k_x\|_{\cH^p} \ge k(x,x)^{1/q}$,
  \end{enumerate}
\end{proposition}
\begin{proof}
  For each item, the first equality follows from duality; see Theorem \ref{thm:dual} and the discussion at the beginning of the section.

  (a) This follows from the fact that the Hankel operator with symbol $k_x$ is the rank
  one operator given by $H_{k_x} (f) = \langle f, k_x \rangle \ol{k_x}$.

  (b)
  Let $\theta = \frac{q-1}{q}$. By reiteration, $\mathcal{H}^q = [ \mathcal{H}^1, \Han]_{\theta}
  = [\mathcal{H}^2, \Han]_{s}$, where $s = 2 \theta - 1 = 1 - \frac{2}{q}$.
  Since $\|k_x\|_{\cH^2} = k(x,x)^{1/2}$ and $\|k_x\|_{{\rm Han}} = k(x,x)$ by part (a),
  interpolation therefore yields
  \begin{equation*}
    \|k_x\|_{\mathcal{H}^q} \le \|k_x\|^{1-s}_{\mathcal{H}^2} \|k_x\|_{\Han}^s = k(x,x)^{1 - 1/q}
  \end{equation*}
  for $2 \le q < \infty$.

  (c) By part (b),
  \begin{equation*}
    \|\delta_x\|_{(\cH^q)^*}
    \ge \frac{k(x,x)}{ \|k_x\|_{\cH^q}} \ge k(x,x)^{1/q}. \qedhere
  \end{equation*}
  \end{proof}

In Section~\ref{secc} we shall prove that the estimates are sharp (up to a constant) in complete Pick spaces.

\section{Complete Pick Spaces}
\label{secc}

Pick's theorem \cite{pi16} gives necessary and sufficient conditions to solve an interpolation problem
in the multiplier algebra of $H^2$ (which is $H^\infty$). It generalizes to matrix-valued functions.
A Hilbert space in which this matrix-valued Pick theorem is true is called a complete Pick space (see the next paragraph for a formal definition). Examples include the Dirichlet space \cite{marsun}, the Sobolev space \cite{ag90b}, and various
Besov spaces on the ball \cite{ahmrRad}; see also \cite{ampi}.

If $\Lambda \subseteq X$, we define $\M_\Lambda$ to be the closed linear span
of the kernel functions $\{ k_\lambda : \lambda \in \Lambda \}$, and let $P$ be the orthogonal projection from $\M$ onto $\M_\Lambda$.
Define $\pi : \Mult (\M ) \to B(\M_\Lambda)$ by
$\pi( \phi) = P M_\phi P$, where $M_\phi$ is multiplication by $\phi$.
Then $\pi$ is always a contractive homomorphism. If it is an exact quotient map
(i.e. it maps the closed unit ball onto the closed unit ball), then $\M$ is said to have the Pick property.
If it is a complete exact quotient map (the induced map on matrices is always an exact quotient map),
then $\M$ is said to have the complete Pick property.

The Drury-Arveson space is the Hilbert function space on the open unit ball $\B_d$ in $\C^d$ or $\ell^2(d)$ with kernel
\[
S(z,w) \= \frac{1}{1 - \la z, w \ra } .
\]
We say a reproducing kernel Hilbert space $\M$ on $X$ is
{\em normalized} if for some choice of base-point $x_0$, we have
$k(x_0, y) = 1$ for all $y$.

The Drury-Arveson space is a normalized space with the complete Pick property,
and every normalized
space with the complete Pick property can be embedded in it \cite{agmc_cnp}, in the sense that there is a function $b: X \to \B_d$ for
some cardinal $d$ so that
\[
k(x,y) \= S(b(x), b(y) ).
\]
Normalized complete Pick spaces always satisfy assumption (A), as the kernel functions are multipliers.
We shall prove that for complete Pick spaces, the inequalities in Proposition \ref{propb2} are equivalences.
For a kernel $k$, let us write $\chp(k)$ to denote the space in \eqref{eqa2} corresponding to the reproducing kernel Hilbert
space with kernel $k$ (which will be called $\cht(k)$ in this notation).
We also write $\Han(k)$ in place of $\Han$ when we need to specify the kernel.

\begin{remark}
In general, for any complete Pick space $\M$, we have
\begin{equation}
\label{eqxm2}
{\rm Han}(\M) =
{\mathcal X}(\M)  ,
\end{equation}
where ${\mathcal X}(\M)$ is defined in Equation \eqref{eqxm}.
Indeed, equality \eqref{eqxm2} was proved in \cite[Thm 2.6]{ahmrWP} under the
hypothesis that $\M$ has the column-row property, and recently in
\cite{Hartz20} it was shown that this property holds in all complete Pick spaces.
%A Hilbert function space $\M$ has the column row property if, whenever
%$\phi_j$ is a sequence of multipliers so that the column
%\[
%\begin{pmatrix}
%  M_{\phi_1} \\
%  M_{\phi_2} \\
%\vdots \\
%\end{pmatrix}
%\ : \ \M \to \M \otimes \ell^2
%\]
%is bounded, then the row
%\[
%    (M_{\phi_1} \ M_{\phi_2} \ \cdots ) :  \M \otimes \ell^2  \to \M
%\]
%is also bounded.
%For any complete Pick space $\M$ satisfying the column-row property, we have
%\begin{equation}
%\label{eqxm2}
%{\rm Han}(\M) =
%{\mathcal X}(\M)  ,
%\end{equation}
%where ${\mathcal X}(\M)$ is defined in Equation \eqref{eqxm}.
%Indeed, this was proved in \cite[Thm 2.6]{ahmrWP}.
  \end{remark}

\begin{lemma}
  \label{lem:comp}
  Let $k,\ell$ be reproducing kernels on $X,Y$ respectively, and let $\varphi: Y \to X$ be a function.
  If the composition operator
  \begin{equation*}
    C_\varphi: \cht(k) \to \cht(\ell), \quad f \mapsto f \circ \varphi,
  \end{equation*}
  is well defined and bounded, then $C_\varphi$ also maps $\chp(k)$ to $\chp(\ell)$ for $1 \le p \le 2$
  and
  \begin{equation*}
    ||C_\varphi||_{\chp(k) \to \chp(\ell)} \le ||C_\varphi||_{\cht(k) \to \cht(\ell)}^{2/p}.
  \end{equation*}
\end{lemma}

\begin{proof}
  It suffices to show the statement for $p=1$. The desired result then follows  from
  complex interpolation (and the observation that a bounded operator between two Banach function
  spaces that acts by composition on a dense subset acts by composition everywhere).

  To show the statement for $p=1$, let
  $f \in \cho(k)$ with $\|f\|_{\mathcal{H}^1(k)} < 1$. Then there exist $g_n,h_n \in \cht(k)$ so that
  $f = \sum_{n} g_n h_n$ and
  \begin{equation*}
    \sum_{n} ||g_n||_{\cht(k)} \, ||h_n||_{\cht(k)} < 1.
  \end{equation*}
  Hence
  \beq
    ||f \circ \varphi||_{\cho(\ell)} &\ \le \ & \sum_n ||g_n \circ \varphi||_{\cht(\ell)} \, ||h_n \circ \varphi||_{\cht(\ell)} \\
   & \le &||C_\varphi||^2_{\cht(k) \to \cht(\ell)},
  \eeq
  which completes the proof.
\end{proof}

The following lemma shows that the classical $H^p$ spaces on the disc can be embedded into
the $\chp$ spaces corresponding to the Drury--Arveson space. In particular, this gives
concrete examples of functions in these spaces on the ball.
In the sequel, let $s$ denote the Szeg\H{o} kernel on the disc and let $S$ be the Drury--Arveson
kernel on $\mathbb{B}_d$.

\begin{lemma}
  \label{lem:Hp_embedding}
  Let $d$ be a cardinal number, let $v \in \ell^2(d)$ be a unit vector and let
  \begin{equation*}
    Q: \ell^2(d) \to \C, \quad z \mapsto \langle z,v \rangle.
  \end{equation*}
  Then for every $1 \le p < \infty$, the map
  \begin{equation*}
 \chp(s) \to \chp(S), \quad f \mapsto f \circ Q,
  \end{equation*}
  is an isometry onto a complemented subspace of $\chp(S)$.
  Similarly, the map
  \begin{equation*}
    \Han(s) \to \Han(S), \quad f \mapsto f \circ Q,
  \end{equation*}
  is an isometry onto a complemented subspace of $\Han(S)$.
\end{lemma}

\begin{proof}
  Since $V: f \mapsto f \circ Q$ is an isometry from $H^2$ into $H^2_d$, it follows from Lemma \ref{lem:comp}
  that it is also a contraction from $\mathcal{H}^1(s)$ into $\mathcal{H}^1(S)$.
  Consider the embedding
  \begin{equation*}
    i: \bD \to \bB_d, \quad \lambda \mapsto \lambda v.
  \end{equation*}
  Then $R: f \mapsto f \circ i$, being the adjoint of $V$, is a co-isometry from $\cht_d$ onto $H^2$.
  Applying Lemma \ref{lem:comp} again, we find that $R$ is a contraction from $\mathcal{H}^1(S)$ into $\mathcal{H}^1(s)$.

  Next, we use duality to prove the statement about $\Han$.
  Let $h \in \Han(s)$ and let $f \in H^2_d \subset H^2_d \odot H^2_d$.
  Then
\beq
    |\langle f, h \circ Q \rangle_{H^2_d}| & \ = \ & |\langle f \circ i, h \rangle_{H^2}| \\
    &\le & ||f \circ i||_{H^2 \odot H^2} \, ||h||_{\Han(s)}\\
  &  \le & ||f||_{H^2_d \odot H^2_d} \, ||h||_{\Han(s)}.
\eeq
  Thus, $h \circ Q$ induces a bounded functional
  on $H^2_d \odot H^2_d$ of norm at most $||h||_{\Han(s)}$, so that $h \circ Q \in \Han(S)$ and the
  map $V: h \mapsto h \circ Q$ is a contraction from $\Han(s)$ to $\Han(S)$.
  Similarly, one checks that $R: f \mapsto f \circ i$ is a contraction from $\Han(S)$ to $\Han(s)$.

  Interpolation therefore shows that $V: h \mapsto h \circ Q$ is a contraction
  from $\mathcal{H}^p(s)$ into $\mathcal{H}^p(S)$ for $1 \le p < \infty$
  and that $R: h \mapsto h \circ i$ is a contraction from $\mathcal{H}^p(S)$ into $\mathcal{H}^p(s)$ for $1 \le p < \infty$.
  Clearly, $R \circ V$ is the identity on $\mathcal{H}^p(s)$ and on $\Han(s)$, hence $V$ is an isometry from $\mathcal{H}^p(s)$
  into $\mathcal{H}^p(S)$ and from $\Han(s)$ to $\Han(S)$.
  Moreover, $V \circ R$ is a projection onto the range of $V$, hence that space is complemented.
\end{proof}

Let $\hha$ denote the half-plane $\{ z \in \C : \Re(z) > \frac{1}{2} \}$.
It is the image of the unit disk under the map $ \zeta \mapsto \frac{1}{1 - \zeta}$.

\begin{lemma}
  \label{cor:kernel_da}
Let $1 \leq p < \infty$, let $w \in \bB_d$ and let $r = ||w||$. If $h$ is an analytic
  function on $\hha$ such that $h \circ s_r \in \chp(s)$, then $h \circ S_w \in \chp(S)$ and
  \begin{equation*}
    ||h \circ S_w||_{\chp(S)}
    =  ||h \circ s_r||_{\chp(s)}.
  \end{equation*}
\end{lemma}

\begin{proof}
  Let $v$ be a unit vector in $\ell^2(d)$ such that $r v = w$
  (i.e. $v = w / r$ if $r \neq 0$ and $v$ is an arbitrary unit vector if $r = 0$) and let
  $Q(z) = \langle z,v \rangle$.
  We apply the isometry of Lemma \ref{lem:Hp_embedding} to $f = h \circ s_r$. Since
  \begin{equation*}
    (s_r \circ Q)(z) = \frac{1}{1 - r \langle z,v \rangle} = S_w(z),
  \end{equation*}
  we have $h \circ s_r \circ Q = h \circ S_w$, and the result follows from Lemma \ref{lem:Hp_embedding}.
\end{proof}

Using universality of the Drury--Arveson space, we can also construct explicit examples of functions in $\mathcal{H}^p(k)$
for complete Pick kernels $k$, at least for $1 \le p \le 2$.

\begin{proposition}
  \label{prop:kernel_containment}
  Let $k$ be a normalized complete Pick kernel on $X$ and let $1 \le p \le 2$. Let $x \in X$
  and set
  \[
  r \ =\ \sqrt{1 - \frac{1}{k(x,x)}}.
  \]
  If $h$ is an analytic
  function on $\hha$ such that $h \circ s_{r} \in \chp(s)$,
  then $h \circ k_x \in \chp(k)$ and
  \begin{equation*}
    ||h \circ k_x||_{\chp(k)} \le ||h \circ s_{r}||_{\chp(s)}.
  \end{equation*}

  In particular, $k_x^{2/p} \in \chp(k)$
  and $||k_x^{2/p}||_{\chp(k)} \lesssim k(x,x)^{1/p}$ for all $x \in X$, where the implied constant  depends
only on $p$, not on $k$ or $x$.
\end{proposition}

\begin{proof}
There exists a function $b: X \to \bB_d$ for a suitable cardinal $d$ such that $k(x,y) = S(b(x),b(y))$
  and so that $f \mapsto f \circ b$ is a co-isometry from $H^2_d$ onto $\cht(k)$.
  Let $w = b(x)$ and note that $r = ||w||$.
  In this setting, Lemma \ref{cor:kernel_da} implies that $h \circ S_w \in \chp(S)$ with
  $||h \circ S_w||_{\chp(s)} = ||h \circ s_r||_{\chp(s)}$.
  By Lemma \ref{lem:comp}, the map $f \mapsto f \circ b$ is a contraction from $\chp(S)$ into $\chp(k)$,
  hence $h \circ S_w \circ b \in \chp(k)$ with $||h \circ S_w \circ b||_{\chp(k)} \le ||h \circ s_r||_{\chp(s)}$.
  Since
  \begin{equation*}
    (S_w \circ b)(y) = \frac{1}{1 - \langle b(y),w \rangle} = k_x(y)
  \end{equation*}
  for each $y \in X$, we have $S_w \circ b = k_x$, which completes the proof of the first statement.

  To prove the additional statement, we let $h(\lambda) = \lambda^{2/p}$.
  Then it follows from the fact that $\chp(s) = H^p$ isomorphically that
  \begin{equation}
  \label{eq4}
    ||s_r^{2/p}||_{\chp(s)} \approx ||s_r||_{H^2}^{2/p} = s(r,r)^{1/p},
  \end{equation}
  so by the first paragraph, $k_x^{2/p} \in \chp(k)$ and
  \begin{equation*}
    ||k_x^{2/p}||_{\chp(k)} \lesssim s(r,r)^{1/p} = k(x,x)^{1/p}. \qedhere
  \end{equation*}
\end{proof}

We can now give asymptotic bounds on $\| \delta_x \|$ for every $1 \leq p \leq \infty$.
For $p=1$, we already proved (b) in Proposition \ref{propb2}; we include it for completeness.
For $p=\infty$ we have $\| \delta_x \|_{{\rm Mult}} = 1$. (Indeed, $1$ is an upper bound because
$M_\phi^* k_x = \overline{\phi(x)} k_x$, and it is attained since the constants are multipliers).

In \cite{arsw11}, Arcozzi, Rochberg, Sawyer and Wick studied the weak product of the Dirichlet space with itself, and for
that space proved (b) and (c) in Theorem \ref{thmc1}, and moreover showed that
$\|k_x \|_{\cho} \approx 1 + \log(k(x,x))$.

\bt
\label{thmc1}
Let $k$ be a normalized complete Pick kernel on $X$. Then for all $1 < p < \infty$
and $\frac{1}{p} + \frac{1}{q} = 1$,
  \begin{enumerate}[label=\normalfont{(\alph*)},resume]
    \item $\| \delta_x\|_{(\chp)^*} = \|k_x\|_{\chq} \approx k(x,x)^{1/p}$,
    \item $\| \delta_x\|_{(\cH^1)^*} = \| k_x\|_{{\rm Han}} = k(x,x)$,
    \item $\| \delta_x \|_{{\rm Han}_0^*} = \|\delta_x\|_{{\rm Han}^*} = \|k_x \|_{\cho} \lesssim 1 + \log(k(x,x))$,
\item for the Drury-Arveson kernel $S$,
 $\| \delta_x \|_{{\rm Han}_0^*}  \approx 1 + \log(S(x,x))$,
  \end{enumerate}
  where the implied constants
do not  depend
 on $k$ or $x$.
\et
\begin{proof}
  The equalities in (a) and (c) follow from the $\mathcal{H}^p$--$\mathcal{H}^q$ duality (Theorem \ref{thm:dual})
  and the $\Han_0$--$\mathcal{H}^1$ and the $\mathcal{H}^1$--$\Han$ dualities.

Suppose first that $1 \le q \le 2$. Then by Proposition \ref{prop:kernel_containment}
and the equality $H^p(s) = H^p$ with equivalent norms,
  \begin{equation}
    \label{eqn:kernel_integral}
    \|k_x\|_{\chq(k)}^q \lesssim \|s_r\|_{H^q}^q
    = \frac{1}{2 \pi} \int_{\T} \frac{1}{|1 - r e^{i t} |^q} \, dt,
  \end{equation}
  where $r = (1 - k(x,x)^{-1})^{1/2}$. The integral above
  behaves like $(1-r^2)^{1-q}$ if $q > 1$ and like $1 + \log( 1/ (1 - r^2))$ if $q=1$; see \cite[Theorem 1.12]{zhu}.
  Hence
  \begin{equation}
    \label{eqn:kernel_integral_estimate}
    \|k_x\|_{\cH^q} \lesssim k(x,x)^{(q-1)/q} = k(x,x)^{1/p}
  \end{equation}
  if $1 < q \le 2$ and
  \begin{equation*}
    \|k_x\|_{\cH^1} \lesssim 1+ \log(k(x,x)).
  \end{equation*}
  This proves (c) and also the inequality $\|k_x\|_{\mathcal{H}^q} \lesssim k(x,x)^{1/p}$
  in (a) if $2 \le p < \infty$.
  The reverse inequality was established in part (c) of Proposition \ref{propb2},
  so (a) holds for $2 \le p < \infty$.

  Next, let $1 < p \le 2$. Part (b) of Proposition \ref{propb2} shows that
  $\|\delta_x\|_{(\mathcal{H}^p)^*} \le k(x,x)^{1/p}$. On the other hand,
  by \eqref{eqn:kernel_integral_estimate}, we have
  \begin{equation*}
    \|\delta_x\|_{(\chp)^*} \ge \frac{k(x,x)}{ \|k_x\|_{\chp}}
    \gtrsim \frac{k(x,x)}{k(x,x)^{1 - 1/p}} = k(x,x)^{1/p},
  \end{equation*}
  so (a) also holds for $1 < p \le 2$.

  Finally, when $k=S$,  Lemma \ref{cor:kernel_da} shows that the estimate in \eqref{eqn:kernel_integral} is actually
  an equivalence, which gives (d).
\end{proof}
\begin{corollary}
If  $k$ is a normalized complete Pick kernel on $X$ and  $\sup_{x} k(x,x) = \infty$, then for
 $1 \leq p < q < \infty$ the containment
$\chq(k) \subseteq \chp(k)$ is strict.
\end{corollary}

\begin{remark}
If $f$ is in $\chp$ for $1 \leq p < \infty$, then one has the pointwise estimate
$|f(x)| \leq \| f \| \| \delta_x \|_{(\chp)^*}$. But since linear combinations of kernel functions are dense
in $\chp$ by Corollary \ref{cor:dense} and each individual kernel function is bounded, one can improve this, for each fixed $f$,  to
\[
|f(x)| \ = \ o( \| \delta_x \|_{(\chp)^*}), \quad k(x,x) \to \infty .
\]
Similarly, if $f \in {\rm Han}_0$ one gets
\[
|f(x)| \ = \ o( \| k_x \|_{\cho}), \quad \|  k_x \|_{\cho} \to \infty .
\]
\end{remark}

Below, we will provide an example to show that in general complete Pick spaces, the estimate in part (c) of Theorem \ref{thmc1}
need not be an equivalence.

Recall that if $\mathcal{M}$ is a reproducing kernel Hilbert space on $X$, then a sequence $(x_n)$ in $X$ is said to be an interpolating
sequence if the evaluation map
\[
E: \varphi \mapsto (\varphi(x_n))
\]
maps $  \Mult(\mathcal{M})$ onto  $\ell^\infty $.
%See \cite[Chapter 9]{ampi} for background.

If $\alpha_n$ is a sequence of positive numbers, and $p \geq 1$,
we let $\ell^p(\alpha_n)$ denote the
Banach sequence space with norm
\[
\| (c_n) \| := \left( \sum_{n=1}^\infty |c_n|^p \alpha_n \right)^{1/p} .
\]
\begin{defin}
The sequence $(x_n)$ is an interpolating sequence for $\chp$ if the evaluation map $E$
maps $\chp$ into and onto $\ell^p( 1/\| \delta_{x_n} \|^p_{(\chp)^*})$.
\end{defin}

The closed graph theorem shows that if $(x_n)$ is an interpolating sequence for $\mathcal{H}^p$, then $E$ is a bounded
map from $\mathcal{H}^p$ onto $\ell^p( 1/\| \delta_{x_n} \|^p_{{\chp}^*})$, so by the open mapping theorem,
the induced map $\mathcal{H}^p / \ker(E) \to \ell^p( 1/\| \delta_{x_n} \|^p_{{\chp}^*})$ has a bounded
inverse. The norm of the inverse is usually called the constant of interpolation.

Shapiro and Shields \cite{shashi61} showed that in the case of the Hardy space of the disc, the interpolating
sequences for $H^p$ are the same for $1 \le p \le \infty$.
It was observed by Marshall and Sundberg that if $\M$ is a complete Pick space, the interpolating
sequences for $\M = \cht $ and $\Mult(\M)$ are the same \cite{marsun}.
% (see \cite[Chapter 9]{ampi} for an exposition).
In \cite{arsw11}, it was shown that for the Dirichlet space, the interpolating sequences for $\cho$ and $\mathcal{H}^2$
are also the same.
Their proof carries over to any complete Pick space.
We first prove the easy implication, which is valid without the complete Pick assumption.
If $k$ is a kernel on $X$ and $V \subset X$,
we let $k \big|_V$ denote the restriction of $k$ to $V \times V$.
Thus, $\mathcal{H}^2(k|_V)$ is a space of functions on $V$.

\begin{lemma}
  \label{lem:IS_H^2_H^1}
  Let $\mathcal{M}$ be a reproducing kernel Hilbert space on $X$ with kernel $k$ satisfying assumption (A)
  and let $(x_n)$ be an interpolating sequence for $\mathcal{H}^2$.
  \begin{enumerate}[label=\normalfont{(\alph*)}]
    \item The sequence $(x_n)$ is an interpolating sequence for $\mathcal{H}^1$.
    \item The evaluation map $E: \mathcal{H}^1 \to \ell^1(1 / \|\delta_{x_n}\|_{{\mathcal{H}^1}^*})$ has a bounded linear right-inverse.
    \item If $V = \{x_n: n \in \mathbb{N}\}$, then
  \begin{equation*}
    \|h\|_{\mathcal{H}^1(k|_V)} \approx \sum_{n} \frac{|h(x_n)|}{k(x_n,x_n)}.
  \end{equation*}
  for all $h \in \mathcal{H}^1(k|_V)$.
  \end{enumerate}
\end{lemma}

\begin{proof}
  By Proposition \ref{propb2} (a), we have $\|\delta_x\|_{(\mathcal{H}^1)^*} = k(x,x)$.
  Lemma \ref{lem:comp} shows that the restriction map $R: \mathcal{H}^1 \to \mathcal{H}^1(k|_V)$
  is contractive. Thus, the lemma will be proved once we show that
  \begin{enumerate}
    \item the map $E_V: h \mapsto (h(x_n))$ maps $\mathcal{H}^1(k |_V)$ boundedly into $\ell^1(1/k(x_n,x_n))$, and
    \item there is a bounded linear operator $T: \ell^1(1/k(x_n,x_n)) \to \mathcal{H}^1$
      so that $E \circ T$ is the identity.
  \end{enumerate}
  Indeed, this follows from the commutative diagram
  \begin{equation*}
  \begin{tikzcd}
    \mathcal{H}^1 \arrow[rd,"E"'] \arrow[r,"R"] & \mathcal{H}^1(k|_V) \ar[d,"E_V"] \\
                                           & \ell^1(\frac{1}{k(x_n,x_n)})
  \end{tikzcd}
  \end{equation*}

  and injectivity of $E_V$.

  To show (1), let
  $h \in \mathcal{H}^1(k|_V)$ with $\|h\|_{\mathcal{H}^1(k|_V)} < 1$.
  Then there exist $f_j,g_j \in \mathcal{M}$ so that
  $h = \sum_j f_j|_V g_j|_V$ and $\sum_j \|f_j\|_{\mathcal{M}} \|g_j\|_{\mathcal{M} } < 1$.
  By the Cauchy--Schwarz inequality,
  \beq
    \sum_n \frac{|h(x_n)|}{k(x_n,x_n)}
  & \  \le  \ & \sum_n \sum_j \frac{|f_j(x_n) g_j(x_n)|}{k(x_n,x_n)}\\
  &\le &
    \sum_j \left[ \sum_n \frac{|f_j(x_n) |^2}{k(x_n,x_n)}
      \sum_m \frac{|g_j(x_m) |^2}{k(x_m,x_m)}  \right]^{1/2} \\
    & \le &
 \| E \|^2_{\cht \to \ell^2(1/k(x_n,x_n))} .
\eeq

So $E$ maps $\mathcal{H}^1$ boundedly into $\ell^1(1/k(x_n,x_n))$.

As for (2), observe that since $(x_n)$ is an interpolating sequence for $\mathcal{H}^2$,
the open mapping theorem implies that there exists a sequence $(f_n)$ in $\mathcal{H}^2$
satisfying $f_n(x_k) = \delta_{nk}$ and $\|f_n\|_{\mathcal{H}^2}^2 \lesssim 1/k(x_n,x_n)$,
hence $\|f_n^2\|_{\mathcal{H}^1} \lesssim 1/k(x_n,x_n)$.
Define
\begin{equation*}
  T: \ell^1(1/k(x_n,x_n)) \to \mathcal{H}^1, \quad (w_n) \mapsto \sum_{n} w_n f_n^2.
\end{equation*}
The series converges absolutely in $\mathcal{H}^1$, the operator $T$ is bounded, and $T(w_n) (x_k) = w_k$,
so $E \circ T$ is the identity.
\end{proof}

%If $k$
%is a normalized kernel on $X$, the formula
%\begin{equation*}
%  \dd_k(x,y)=\sqrt{1-\frac{|k(x,y)|^2}{k(x,x)k(y,y)}}
%\end{equation*}
%defines a pseudo-metric on $X$; see \cite[Lemma 9.9]{ampi}.
%
%\begin{lemma}
%  \label{lem:metric_han} Let $k$ be a normalized kernel on $X$ satisfying assumption (A).
%  For $x \in X$ define $b_x=\frac{k_x}{k(x,x)}$. Then for all $x,y\in X$ we have $\dd_k(x,y)=\|b_x-b_y\|_{\Han}$.
%\end{lemma}
%
%\begin{proof} If $u_x=\frac{k_x}{\sqrt{k(x,x)}}$, then
%  $H_{b_x}$ is the rank one operator $f \mapsto \langle f,u_x \rangle \overline{u_x}$.
%  Thus, if $P_{u_x}$ denotes the orthogonal projection from $\mathcal{H}$ onto $\mathbb{C} u_x$, i.e.
%  $P_{u_x} f = \langle f,u_x \rangle u_x$, then
%  \begin{equation*}
%    \|b_x - b_y\|_{\Han} = \|H_{b_x} - H_{b_y}\|_{B(\mathcal{H}^2, \overline{\mathcal{H}^2})}
%    = \|P_{u_x} - P_{u_y}\|_{B(\mathcal{H}^2)} = \dd_k(x,y),
%  \end{equation*}
%\textbf{TODO: I don't see this any more.}
%  where the last identity follows from the general
%  identity $\|P_u - P_v\| = \sqrt{1 - |\langle u,v \rangle |^2}$,
%which is valid for all unit vectors $u,v$; see \cite[Proposition 5]{ARS+11}.
%\end{proof}

\begin{lemma}
\label{leme5}
Let $u$ and $v$ be unit vectors in a Hilbert space $\M$.
Let $\tau$ and $\omega$ be complex numbers satisfying $|\tau| = |\omega|$ and
\be
\label{eqe51}
\tau \la v, u \ra \= \omega \la u, v \ra .
\ee
Then
\begin{equation}
\label{eqe52}
\nonumber
\| \rankone{u}{\overline{u}} - \omega \rankone{v}{\overline{v}} \|_{\M \to \overline{\M}}
\=
\|  \rankone{u}{u}- \tau  \rankone{v}{v} \|_{\M \to \M} .
\end{equation}
\end{lemma}
\bp
Let $C : \overline{\M} \to \M$ be the anti-linear isometric operator $ \bar f \mapsto f$.
The norm squared of $ C ( \rankone{u}{\overline{u}} - \omega \rankone{v}{\overline{v}})$ is
\begin{equation}
\label{eqe53}
\sup_{ \| f \|_\M = 1} \| \la u, f \ra u - \bar \omega \la v , f \ra v \|^2 .
\end{equation}
The norm squared of $\rankone{u}{u} - \tau \rankone{v}{v}$ is
\begin{equation}
\label{eqe54}
\sup_{ \| f \|_\M = 1} \| \la f, u \ra u - \tau \la f , v \ra v \|^2 .
\end{equation}
Expanding \eqref{eqe53} and \eqref{eqe54} and using \eqref{eqe51}, we see that they are equal.
\ep

If $k$
is a normalized kernel on $X$, the formula
\begin{equation*}
  \dd_k(x,y)=\sqrt{1-\frac{|k(x,y)|^2}{k(x,x)k(y,y)}}
\end{equation*}
defines a pseudo-metric on $X$; see \cite[Lemma 9.9]{ampi}.

\begin{lemma}
  \label{lem:metric_han} Let $k$ be a normalized kernel on $X$ satisfying assumption (A).
  For $x \in X$ define $b_x=\frac{k_x}{k(x,x)}$. Let $x,y \in X$ and let $\omega$ be a unimodular complex
  number satisfying $\langle k_y,k_x \rangle = \omega \langle k_x,k_y \rangle$. Then $\dd_k(x,y)=\|b_x- \overline{\omega} b_y\|_{\Han}$.
\end{lemma}

\begin{proof} If $u_x=\frac{k_x}{\sqrt{k(x,x)}}$, then
  $H_{b_x}$ is the rank one operator $f \mapsto \langle f,u_x \rangle \overline{u_x}$.
  Thus, if $P_{u_x}$ denotes the orthogonal projection from $\mathcal{H}$ onto $\mathbb{C} u_x$, i.e.
  $P_{u_x} f = \langle f,u_x \rangle u_x$, then by Lemma \ref{eqe51}, we have
  \begin{equation*}
    \|b_x - \overline{\omega} b_y\|_{\Han} = \|H_{b_x} - \omega H_{b_y}\|_{\mathcal{H}^2 \to \overline{\mathcal{H}^2}}
    = \|P_{u_x} - P_{u_y}\|_{\mathcal{H}^2 \to \mathcal{H}^2}.
  \end{equation*}
  As $u_x$ and $u_y$ are unit vectors, we have
  \begin{equation*}
    \|P_{u_x} - P_{u_y} \| = \sqrt{1 - | \langle u_x,u_y \rangle|^2} = \dd_k(x,y)
  \end{equation*}
  (this can be seen by observing that $P_{u_x} - P_{u_y}$ is a trace $0$, rank $2$ self-adjoint operator,
with determinant $|\la u_x, u_y\ra|^2 - 1$ on the two-dimensional space spanned by $u_x$ and $u_y$).
\end{proof}

We now establish the announced equality of interpolating sequences for $\mathcal{H}^1$ and $\mathcal{H}^2$
in the setting of complete Pick spaces.

\bt
\label{thmIS}
Let $k$ be a normalized complete Pick kernel on $X$. Then the interpolating sequences for
$\cho$ and $\cht$ coincide.
\et
\bp
In light of Lemma \ref{lem:IS_H^2_H^1}, it remains to show that every interpolating sequence
$(x_n)$ for $\mathcal{H}^1$ is interpolating for $\mathcal{H}^2 = \mathcal{M}$.
To this end, by \cite{ahmrIS}, we need to show that
\begin{enumerate}
  \item  the sequence is weakly separated, which means there exists $\delta > 0$ so that for $m \neq n$ we have
\[
  \dd_k(x_m,x_n) \ge \delta,
\]
and
\item
 it satisfies the Carleson measure condition, namely there exists some constant $C$ so that
 \[
\sum_{n=1}^\infty\frac{ |f(x_n)|^2 }{k(x_n, x_n)} \ \leq \ C \|f\|^2 \qquad \forall \ f \in \M.
\]
\end{enumerate}
As $ f^2 \in \cho$ with $\| f^2 \|_{\cho}  \leq \| f \|^2_\M$, we get (2) immediately.

To see weak separation, we use the open mapping theorem to find a constant $c > 0$ so that for any
distinct points $x_m, x_n$,
there is a function $h$ of norm at most $1$ in $\cho$
with $h(x_m) = c k(x_m, x_m)$ and $h(x_n) = -c \bar \omega k(x_n, x_n)$,
where $\omega \la k_{x_m} , k_{x_n} \ra = \la k_{x_n} , k_{x_m} \ra$.
This means
\be
\label{eqe41}
\big\| \frac{1}{k(x_m,x_m)} \delta_{x_m} -
\frac{\omega}{k(x_n,x_n)} \delta_{x_n}  \big\|_{(\mathcal H^1)^*} \ \geq \ 2c .
\ee
Write $b_x = k_x / k(x,x)$.
Using the anti-linear $\mathcal{H}^1$--$\Han$ duality, and Lemma \ref{lem:metric_han}, the left-hand side of \eqref{eqe41}
is equal to
\begin{equation*}
  \|b_{x_m} - \overline{\omega} b_{x_n}\|_{\Han} = d_k(x_m,x_n),
\end{equation*}
hence $(x_n)$ is weakly separated.
\ep
%To see weak separation, we use the fact there is a constant $c > 0$ so that for any
%distinct points $x_m, x_n$ in $V$,
%and any unimodular $\omega$, there is a function $h$ of norm at most $1$ in $\cho$
%with $h(x_m) = c k(x_m, x_m)$ and $h(x_n) = - \omega^2 c k(x_n, x_n)$.
%Write $h = \sum f_j g_j$, with $\sum \| f_j\|^2_\M $ and $\sum \| g_j\|^2_\M $ both less than or
%equal to $1$.
%We have on one hand that
%\be
%\label{eqe33}
%h(x_m)  - \bar \omega^2 \frac{k(x_m,x_m)}{k(x_n,x_n)} h(x_n)
%\=
%2c k(x_m,x_m) .
%\ee
%We also have
%\beq
%&
%\left| f_j(x_m) g_j(x_m) - \bar \omega  \sqrt{\frac{k(x_m,x_m)}{k(x_n,x_n)}}f_j(x_n) g_j(x_m) \right.\\
%&\left. +  \bar \omega  \sqrt{\frac{k(x_m,x_m)}{k(x_n,x_n)}}f_j(x_n) g_j(x_m)
%- \bar \omega^2  \frac{k(x_m,x_m)}{k(x_n,x_n)}f_j(x_n) g_j(x_n) \right|\\
%&\leq 2 \sqrt{k(x_m,x_m)} \| k_{x_m} - \bar \omega  \sqrt{\frac{k(x_m,x_m)}{k(x_n,x_n)}} k_{x_n} \|
%\| f_j \| \| g_j\|
%\eeq
%Summing over $j$, using Cauchy-Schwarz, and comparing with \eqref{eqe33} we conclude
%\[
%\| k_{x_m} - \bar \omega  \sqrt{\frac{k(x_m,x_m)}{k(x_n,x_n)}} k_{x_n} \|
%\ \geq \ c \| k_{x_m} \| .
%\]
%This give \eqref{eqe36} with $\delta = c^2 /2$.

  We now give an example of a normalized complete Pick space with unbounded kernel in which every function in $\Han$ is bounded.

\begin{example}
\label{exam316}
{\rm
  Let $(e_n)$ denote the standard basis of $\ell^2$.
  Let $(r_n)$ be a sequence in $[0,1)$ tending to $1$ with the properties that $r_0 = 0$, the sequence $x_n = r_n e_n$ is an interpolating
  sequence for $\Mult(H^2_\infty)$ and $\sum_n (1 - r_n^2) < \infty$
  (this can be done for instance by \cite[Proposition 5.1]{ahmrIS}; the last property in fact follows from being an interpolating sequence).
  Let $V = \{ r_n e_n: n \in \mathbb{N} \}$ and let $\mathcal{M} = H^2_\infty \big|_V$,
  which is a normalized complete Pick space on $V$ whose kernel $k$
  satisfies $\lim_n k(x_n,x_n) = \infty$.
  We claim that $\sup_x \|\delta_x\|_{\Han} < \infty$.

  By Theorem \ref{thmc1}, $\|\delta_x\|_{\Han} = \|k_x\|_{\mathcal{H}^1}$, so by Lemma \ref{lem:IS_H^2_H^1} (c), it suffices to show that
  \begin{equation*}
    \sup_j \sum_n \frac{|k(x_n,x_j)|}{k(x_n,x_n)} < \infty.
  \end{equation*}
  But $k(x_n,x_j) = \frac{1}{1 - r_n r_j \langle e_n, e_j \rangle } = 1$ if $n \neq j$, hence
  \begin{equation*}
    \sum_n \frac{|k(x_n,x_j)|}{k(x_n,x_n)} = 1 + \sum_{n \neq j} \frac{1}{k(x_n,x_n)}
    \le 1 + \sum_n (1 - r_n^2) < \infty,
  \end{equation*}
  which is independent of $j$.
	}
\end{example}

The preceding example takes place on the ball in infinite dimensions.
In fact, the $\log$ estimate in part (c) of Theorem \ref{thmc1} may not be an equivalence even on the disc, as the following example shows.

\begin{example}
\label{exam317}
{\rm
    Let $(y_n)$ be a strictly increasing sequence in $[1,\infty)$, with $y_0 = 1$, and tending to infinity so fast
  that
  \begin{enumerate}
    \item $x_n: = 1 - \frac{1}{y_n}$ is an interpolating sequence for $H^\infty$,
    \item $\sum_n \frac{y_{n-1}}{y_n} < \infty$, and
    \item $\lim_{n \to \infty} \frac{n}{\log(y_n)} = 0$.
  \end{enumerate}
  For instance, $y_n = 2^{2^n}$ or $y_n = (n!)^2$ will do.
  Let $V = \{ x_n : n \in \mathbb{N} \} \subset \mathbb{D}$ and let $\mathcal{M} = H^2 \big|_V$.
  We will show that
  \begin{equation*}
    \lim_j \frac{\|\delta_{x_j}\|_{\Han}}{\log(k(x_j,x_j))} = 0.
  \end{equation*}
  Again by Theorem \ref{thmc1} and Lemma \ref{lem:IS_H^2_H^1} (c), it suffices to show that
  \begin{equation*}
    \lim_j \frac{1}{\log(k(x_j,x_j))} \sum_n \frac{|k(x_n,x_j)|}{k(x_n,x_n)} = 0.
  \end{equation*}
  To this end, note that
  \begin{equation*}
    \sum_n \frac{|k(x_n,x_j)|}{k(x_n,x_n)}
    = \sum_n \frac{1 - x_n^2}{1 - x_n x_j} \lesssim
    \sum_{n=0}^{j} \frac{1 - x_n^2}{1- x_n x_j} + \sum_{n=j+1}^\infty \frac{1 - x_n}{1 - x_n x_j}.
  \end{equation*}
  Each summand in the first sum in bounded by $1$. In the second sum, note that
  $x_j \le x_{n-1}$, hence
  \begin{equation*}
    \sum_n \frac{|k(x_n,x_j)|}{k(x_n,x_n)} \lesssim
    (j+1) + \sum_{n=j+1}^\infty \frac{1 - x_n}{1 - x_{n-1}}
    \le (j+1) + \sum_{n} \frac{y_{n-1}}{y_n},
  \end{equation*}
  and the last sum converges by Property (2). Thus,
  \begin{equation*}
    \frac{1}{\log(k(x_j,x_j))} \sum_n \frac{|k(x_n,x_j)|}{k(x_n,x_n)}
    \lesssim \frac{j+1}{\log(y_j)} \xrightarrow{j \to \infty} 0
  \end{equation*}
  by Property (3).
}
\end{example}

\begin{remark}
  Suppose that $\mathcal{M}$ is a normalized complete Pick space with kernel
  $k(x,y) = \frac{1}{1 - \langle b(x),b(y) \rangle}$, where $b: X \to \mathbb{B}_d$. Then
  the map $f \mapsto f \circ b$ takes $H^2_d$ to $\mathcal{H}$. By Lemma \ref{lem:comp},
  it also takes $\mathcal{H}^p(S)$ to $\mathcal{H}^p(k)$ for $1 \le p \le 2$, and it is easy to see that it
  maps $\Mult(H^2_d)$ to $\Mult(\mathcal{M})$. The preceding two examples
  show that in general, it does not map $\Han(S)$ to $\Han(k)$, because by part (d) of Theorem \ref{thmc1},
  we have $\|\delta_x\|_{\Han(S)} \approx 1 + \log(S(x,x))$. Similarly, Proposition \ref{prop:kernel_containment}
  does not hold with $\Han$ in place of $\mathcal{H}^p$.
\end{remark}

\section{Questions}
\label{sece}

\begin{question}
{\rm
There are many interesting Hilbert function spaces for which assumption (A) fails,
such as $\ell^2$, the Hardy space of the upper half-plane, and the Fock space.
One can still define an $\chp$ scale for these spaces for $p \in [1,2]$ by interpolating
between $\mom$ and $\M$. Is there a general method to identify the anti-duals of these
spaces with Banach function spaces on $X$?
}
\end{question}

\begin{question}
{\rm
How does one define the $\chp$ scale for $0 < p < 1$?
}
\end{question}

\begin{question}
  \label{quest:mult_int}
{\rm
When can we recover the $\chp$ spaces
isomorphically
by interpolating between $\M \odot \M$ and $\Mult(\M)$?
This is true for the Hardy space \cite{jo83}.
}
\end{question}

\begin{question}
{\rm
What are the multipliers of $\chp$?
When are they the same as $\Mult(\M)$?
For complete Pick spaces, Clou\^atre and the second named author \cite{CHar,Hartz20} show that
$\Mult(\M) = \Mult( \M \odot \M)$. Is this enough to get $\Mult(\chp) = \Mult(\M)$
for $1 \leq p \leq 2$?
}
\end{question}

\begin{question}
  \label{quest:Hpbig}
%Does Proposition \ref{prop:kernel_containment} hold for $p > 2$?
  {\rm
If $k(x,y) = \frac{1}{1 - \langle b(x),b(y) \rangle}$, does the map
$f \mapsto f \circ b$ take $\mathcal{H}^p(S)$ to $\mathcal{H}^p(k)$ for $p > 2$?}
\end{question}

A positive answer to Question \ref{quest:mult_int} would imply a positive answer to Question \ref{quest:Hpbig},
as the map $f \mapsto f \circ b$ takes multipliers to multipliers.

\begin{question}
  {\rm
Are the interpolating sequences for $\chp$ the same as for $\cht$ for $1 < p < 2$?
What about $2 < p < \infty$?}
\end{question}

\begin{question}
  {\rm What is a function theoretic description of $\chp$ for the Dirichlet space or for weighted Besov--Sobolev spaces?}
\end{question}

%\bibliography{../references_uniform_partial}

\begin{bibdiv}
\begin{biblist}

\bib{ag90b}{article}{
      author={Agler, Jim},
       title={{Nevanlinna-Pick interpolation on Sobolev space}},
        date={1990},
     journal={Proc. Amer. Math. Soc.},
      volume={108},
       pages={341\ndash 351},
}

\bib{agmc_cnp}{article}{
      author={Agler, Jim},
      author={McCarthy, John~E.},
       title={Complete {Nevanlinna-Pick} kernels},
        date={2000},
     journal={J. Funct. Anal.},
      volume={175},
      number={1},
       pages={111\ndash 124},
}

\bib{ampi}{book}{
      author={Agler, Jim},
      author={M\raise.45ex\hbox{c}Carthy, John~E.},
       title={Pick interpolation and {H}ilbert function spaces},
      series={Graduate Studies in Mathematics},
   publisher={American Mathematical Society, Providence, RI},
        date={2002},
      volume={44},
        ISBN={0-8218-2898-3},
         url={http://dx.doi.org/10.1090/gsm/044},
      review={\MR{1882259}},
}

\bib{ahmrRad}{incollection}{
      author={Aleman, Alexandru},
      author={Hartz, Michael},
      author={M{\raise.45ex\hbox{c}C}arthy, John~E.},
      author={Richter, Stefan},
       title={Radially weighted {B}esov spaces and the {P}ick property},
        date={2019},
   booktitle={Analysis of operators on function spaces},
      series={Trends Math.},
   publisher={Birkh\"{a}user/Springer, Cham},
       pages={29\ndash 61},
         url={https://doi-org.libproxy.wustl.edu/10.1007/978-3-030-14640-5_3},
      review={\MR{4019466}},
}

\bib{ahmrIS}{article}{
      author={Aleman, Alexandru},
      author={Hartz, Michael},
      author={M\raise.45ex\hbox{c}Carthy, John~E.},
      author={Richter, Stefan},
       title={Interpolating sequences in spaces with the complete {P}ick
  property},
        date={2019},
        ISSN={1073-7928},
     journal={Int. Math. Res. Not. IMRN},
      number={12},
       pages={3832\ndash 3854},
         url={https://doi.org/10.1093/imrn/rnx237},
      review={\MR{3973111}},
}

\bib{ahmrWP}{unpublished}{
      author={Aleman, Alexandru},
      author={Hartz, Michael},
      author={M\raise.45ex\hbox{c}Carthy, John~E.},
      author={Richter, Stefan},
       title={Weak products of complete {P}ick spaces},
        date={2019},
        note={Indiana Math. J., to appear},
}

\bib{arsw11}{article}{
      author={Arcozzi, N.},
      author={Rochberg, R.},
      author={Sawyer, E.},
      author={Wick, B.~D.},
       title={Function spaces related to the {D}irichlet space},
        date={2011},
        ISSN={0024-6107},
     journal={J. Lond. Math. Soc. (2)},
      volume={83},
      number={1},
       pages={1\ndash 18},
         url={https://doi-org.libproxy.wustl.edu/10.1112/jlms/jdq053},
      review={\MR{2763941}},
}

\bib{bl76}{book}{
      author={Bergh, J.},
      author={{L\"ofstr\"om}, J.},
       title={Interpolation spaces},
   publisher={Springer-Verlag},
     address={Berlin},
        date={1976},
}

\bib{cal64}{article}{
      author={Calder\'on, A.P.},
       title={Intermediate spaces and interpolation, the complex method},
        date={1964},
     journal={Studia Math.},
      volume={24},
       pages={113\ndash 190},
}

\bib{CHar}{article}{
      author={Clou\^{a}tre, Rapha\"{e}l},
      author={Hartz, Michael},
       title={Multipliers and operator space structure of weak product spaces},
        date={to appear},
     journal={Anal. PDE},
        note={arXiv:1909.12883},
}

\bib{cs98}{article}{
      author={Cobos, Fernando},
      author={Schonbek, Tomas},
       title={On a theorem by {L}ions and {P}eetre about interpolation between
  a {B}anach space and its dual},
        date={1998},
        ISSN={0362-1588},
     journal={Houston J. Math.},
      volume={24},
      number={2},
       pages={325\ndash 344},
      review={\MR{1690401}},
}

\bib{fs72}{article}{
      author={Fefferman, C.},
      author={Stein, E.~M.},
       title={{$H^{p}$} spaces of several variables},
        date={1972},
        ISSN={0001-5962},
     journal={Acta Math.},
      volume={129},
      number={3-4},
       pages={137\ndash 193},
         url={https://doi-org.libproxy.wustl.edu/10.1007/BF02392215},
      review={\MR{447953}},
}

\bib{Hartz20}{article}{
      author={Hartz, Michael},
       title={Every complete {P}ick space satisfies the column-row property},
     journal={arXiv:2005.09614},
}

\bib{jo83}{article}{
      author={Jones, Peter~W.},
       title={{$L^{\infty }$} estimates for the {$\bar \partial $} problem in a
  half-plane},
        date={1983},
        ISSN={0001-5962},
     journal={Acta Math.},
      volume={150},
      number={1-2},
       pages={137\ndash 152},
         url={https://doi.org/10.1007/BF02392970},
      review={\MR{697611}},
}

\bib{marsun}{unpublished}{
      author={Marshall, D.},
      author={Sundberg, C.},
       title={Interpolating sequences for the multipliers of the {Dirichlet}
  space},
        date={1994},
        note={Preprint; see
  http://www.math.washington.edu/$\sim$marshall/preprints/preprints.html},
}

\bib{pi16}{article}{
      author={Pick, Georg},
       title={{\"U}ber die {Beschr\"ankungen} analytischer {F}unktionen, welche
  durch vorgegebene {F}unktionswerte bewirkt werden},
        date={1916},
     journal={Math. Ann.},
      volume={77},
       pages={7\ndash 23},
}

\bib{pis96}{article}{
      author={Pisier, Gilles},
       title={The operator {H}ilbert space {${\rm OH}$}, complex interpolation
  and tensor norms},
        date={1996},
        ISSN={0065-9266},
     journal={Mem. Amer. Math. Soc.},
      volume={122},
      number={585},
       pages={viii+103},
         url={https://doi.org/10.1090/memo/0585},
      review={\MR{1342022}},
}

\bib{pi03}{book}{
      author={Pisier, Gilles},
       title={Introduction to operator space theory},
      series={London Mathematical Society Lecture Note Series},
   publisher={Cambridge University Press, Cambridge},
        date={2003},
      volume={294},
        ISBN={0-521-81165-1},
         url={http://dx.doi.org/10.1017/CBO9781107360235},
}

\bib{shashi61}{article}{
      author={Shapiro, H.S.},
      author={Shields, A.L.},
       title={On some interpolation problems for analytic functions},
        date={1961},
     journal={Amer. J. Math.},
      volume={83},
       pages={513\ndash 532},
}

\bib{wat00}{article}{
      author={Watbled, Fr\'{e}d\'{e}rique},
       title={Complex interpolation of a {B}anach space with its dual},
        date={2000},
        ISSN={0025-5521},
     journal={Math. Scand.},
      volume={87},
      number={2},
       pages={200\ndash 210},
         url={https://doi.org/10.7146/math.scand.a-14305},
      review={\MR{1795744}},
}

\bib{zhu}{book}{
      author={Zhu, K.},
       title={Operator theory in function spaces},
      series={Monographs and textbooks in pure and applied mathematics},
   publisher={Marcel Dekker, Inc.},
     address={New York},
        date={1990},
      volume={139},
}

\end{biblist}
\end{bibdiv}
%\end{document}


\end{document}